\documentclass[11pt,reqno]{amsart}
\usepackage{amsmath,amssymb, graphics, graphicx, layout,caption, float,hyperref}
\usepackage{color}
\usepackage{multirow}
\usepackage{epstopdf}

\usepackage[text={6.5in,9in},centering]{geometry}

\newtheorem{Lemma}{Lemma}

\newtheorem{Proposition}[Lemma]{Proposition}
\newtheorem{Remark}[Lemma]{Remark}

\newcommand{\eq}[2]{\begin{equation}\begin{split}#1\end{split}\label{#2}\end{equation}}

\begin{document}

\title[Mathematical Modeling of Drug Use]{Mathematical Modeling of Drug Use: The dynamics of monosubstance dependence for two addictive drugs}

\pagestyle{headings}

\author[Colegate]{ Stephen  Colegate}\address{Department of Mathematics, Miami University}  \email{colegasp@miamioh.edu} 
\author[Liu]{Changrui Liu}\address{Department of Mathematics, Miami University}  \email{liuc4@miamioh.edu}

\begin{abstract}{ Based on previous work done in this field, we build a dynamical system that describes changes in drug addiction in an isolated population when two addictive substances are available simultaneously. We then use our model to investigate whether the system captures the process  of users switching drug habits.  One of the motivations  for this project is to mathematically check the conjecture that being addicted to a less-addictive substance will effectively lead individuals to become dependent on more addictive  and potentially more dangerous  drugs. 

We introduce additional assumptions, under which our model is reduced to a competitive Lotka-Volterra system. This dynamical system has three or four  fixed points, stability of which then gives an implication about  the outcomes of the competition between addictive substances and, therefore, the fate of individuals in the population.  From the analysis of the reduced model, we determine
that there actually exist parameter regimes that capture the following dynamics: depending on the initial distribution of the drug preference; either the use of both drugs will die out, the usage of one of the drugs will become prevalent, or addictions to both drugs will coexist.
}\end{abstract}

\maketitle

\section{Introduction}

Drug addiction is a major concern in modern societies. For example, one of the most widely spread legal addictive substances is nicotine. Particularly, prevalence of smoking  is highest among young people \cite{Compton}. Smokers influence nonsmokers in a variety of ways \cite{Chen}, from second-hand and third-hand nicotine exposure \cite{Ballantyne} to depiction of smoking in movies and television. The availability, low costs and  absence of stigma makes it very easy to start nicotine consumption through  smoking. Nonusers often experience peer pressure or some other form of temptation \cite{Compton}. Nicotine is considered an extremely addictive substance, but the danger to the health of users and people around them, although is very well known, is not immediate. On the other hand, most  illegal drugs, which are  highly addictive, put the lives of users and others around them under direct threat. Very often, it only takes one incident to cause addiction \cite{Davies}. Sometimes drug addiction is also viewed as a mental illness problem \cite{Weinberger}. The person may be unable to resist his or her addiction. Some users are successful in breaking their addictions through some form of rehabilitation. However, many are unsuccessful  in their attempts to  break away from the habit no matter how hard they try  \cite{Chen}, \cite{Davies}.  

The first major groundwork in modeling drug addictions is the Epidemic SIR Model. Developed by Kermack and McKendrick, this models an epidemic that affects a population \cite{Brauer}. The population is divided into three categories, namely susceptible, infected, and recovered. The susceptible group includes those who have never been affected with the disease. When an individual contracts the illness, they are then classified as infected. A person who becomes infected may be able to recover with medical treatments or the like. Those who manage to overcome their illness are then classified as recovered. Many epidemics can be tracked using the Epidemic SIR model.

Inspired by Kermack and McKendrick's Epidemic SIR Model, Castillo-Garsow, Jordan-Salivia, and Rodriguez-Herrera constructed a dynamical system to model smoking addictions and relapse \cite{Castillo}. Their model is similar to those developed for other drugs, including cocaine \cite{Caselles}. In their model, addiction to nicotine is treated as an illness. People who are not smoking or do not exhibit symptoms of addiction to nicotine form the susceptible group. When a person from the susceptible group starts showing signs of addiction, they are then classified as an drug addict. If an addict manages to break free from the habit, they then move to the recovered group. The nature and the stability of the fixed points of the system indicate what happens to the population with time. The fixed points show whether or not nicotine addiction is expected to grow or decline.

In this paper, we expand the model in \cite{Castillo} by introducing a second drug. It has been observed \cite{Fergusson} that when a user becomes addicted to one drug, they have a greater chance of becoming addicted to other ``harder" drugs. The introduction to different addictive substances can be initiated by seeking 
to keep the same level of ``high", instantaneous availability of it, or as an attempt to break free from their current habit.  It is therefore desirable to analyze how \emph{two} drugs coexist and interact with each other in a society.
We use this new model to study the dynamics of addiction between two  drugs mathematically, using methods of applied dynamical systems. The modeling approach is similar to one used in \cite{Henneman} to track both influenza and bacterial infections in a population to see how one disease interacts with the other. To summarize, we present a system of differential equations that models drug addicted population dynamics when two addictive substances are available. We analyze the equilibria of this system to examine effects that particular drug users have on the rest of the population.

\section{Model}
\subsection{The Population}

The population is divided into five groups. Let $S = S(t)$ denote the number of individuals who are susceptible to becoming addicted at time $t\geq 0$. In other words, those in group $S$ have never experienced a drug addiction in their lifetime. These include newborns, so the growth of $S$ is related to the population birth rate. 
Let $D_1=D_1(t)$ and $D_2=D_2(t)$ be the people who, at time $t$, are addicted to drug 1 and drug 2, respectively. Those who are in either $D_1$ or $D_2$ are active taking their respective drug. In order to be classified as an active drug user, the person must have used the drug within the past year. Let $R_1$ and $R_2$ represent those who have recovered from using either drug 1 or drug 2, respectively. The distinction of whether a person is classified into $R_1$ or $R_2$ is based on their last drug addiction. Persons who are in one of the recovered groups have not used either drug for at least a year or more.

Let $N$ denote the  total population, that is
\begin{equation}N=S+D_1+D_2+R_1+R_2.\label{1}\end{equation}
In this paper, we assume that the size of the population remains constant. This is justified if the total population $N$ is much larger than the drug-addicted portion of it. A good rule of thumb generically used in population dynamics is to make the size of the population at least $10,000$ or greater to ensure the total population remains fixed with time. 

We make a few additional assumptions about the five different groups in the population. First, we assume that there is no intersection between those in $D_1$ and in $D_2$. This implies that a drug 1 user cannot be addicted to drug 2, so long as the person is addicted to the first drug.  We do not know how realistic this assumption is, as it is clearly possible that any individual can be addicted to one or both drugs. Here, we will assume that the number of such individuals is negligibly small compared to the number of addicts who have a preferred drug.
To get around this difficulty, the definition of $D_1$ and $D_2$ must be carefully developed. We say that a person who is classified in the $D_1$ group not only is a current user of drug 1 but that person prefers to use drug 1 over drug 2. Subjects in $D_1$ can both be actively using drug 1 and drug 2 at the same time, however, their preference is to use drug 1 over drug 2 when given the choice between the two.
We define $D_2$ in a similar manner. A person classified in the $D_2$ group prefers using drug 2 over drug 1. By defining $D_1$ and $D_2$ in this way, we avoid misinterpreting what it means to be an addict of one or both drugs.

A person who has not used a particular drug for at least one year is classified as recovered from their addiction. Relapse occurs when a person who has been clean from any drug use longer than a year begins their addiction to the same drug again. A person who has recovered from using drug 1 for one year is placed in $R_1$. The person relapses back to $D_1$ if that person becomes addicted to drug 1 again. We assume that a person who relapses from a particular drug resumes use of the same preferred drug.
So in the model, there is no direct link between $D_1$ and $R_2$, nor is there one for $D_2$ and $R_1$. 
If a person is addicted to both drugs but only recovers from one, they are not  classified as recovered but is instead placed in the group who prefer using the other drug. Members of both $R_1$ and $R_2$ may include those who have recovered from both drugs. They are distinguished only by their most recent addiction.

\subsection{The Parameters}

The parameters in the model are identified by which drug they are associated with. Let $i=1$, $2$ for drug 1 and drug 2 respectively. Each of the parameters are assumed constant throughout the entire process. Since all of the parameters are fractional rates, their values must be between 0 and 1, inclusive. It is quite possible in some isolated cases for some parameters to be 0 and may be treated as a special case.  See the Conclusion section for a brief discussion on these special cases. Each of these parameters can either be derived or estimated. See \cite{Castillo} and \cite{Brauer} for discussion on one method of estimating these parameters numerically.

Let $\beta_i$ denote the influence rate of the drug $i$ users per year. The influence rate takes into account any actions that drug users could do that can influence susceptible members into using drug $i$. Influences may include peer pressure, family history, and whether friends and colleagues are using drugs, but we assume that a person who is susceptible to becoming a drug user is only influenced by peer pressure alone. While there may be other reasons why an individual may become addicted to drugs, such as for medical or biological needs or for genetical predisposition, reasons other than peer pressure are not taken into account in these parameters here.

Define $\alpha_1$ to be the rate of those who switch from being a primary drug 2 user to becoming a primary drug 1 user. In a similar manner, let $\alpha_2$ represent the rate of those who become primary drug 2 users who were originally drug 1 users. We can think of these two parameters as the ``switchover rate''. These two parameters give an indication of how often drug users switch their preference. The meaning of $\alpha_1$ and  $\alpha_2$ are similar to those of the $\beta_i$'s. Here, $\alpha_1$ and  $\alpha_2$ are both treated as influence rates. Where they differ in interpretation is the group they influence. While the $\beta_i$'s place peer pressure on those who have never used drugs before, the $\alpha_i$'s represent \emph{current} drug users persuading addicts of the other drug to switch their preference to drug $i$ instead. Since peer pressure is significantly different on those who have never used drugs before than those who are currently addicted, the model distinguishes between these two types of influences.

Denote $\gamma_i$ as the recovery rate for the drug $i$ users per year. This parameter determines how many people recover from $D_i$ to $R_i$ in one year. The size of $\gamma_i$ reflects upon how many active users recover from using drug $i$. The larger $\gamma_i$ gets, the easier it is for a user to quit using drug $i$. Those who recover are classified by their most recent drug addiction. Similarly, let $\delta_i$ be the relapse rate per year for drug $i$ users. Like the recovery rate, the relapse rate determines how many people relapse from $R_i$ to $D_i$ in one year. Individuals that relapse are classified as a drug addict, based solely on their previous drug preference. The larger $\delta_i$ gets, the more users are relapsed back to their previous drug $i$ addiction. Both $\gamma_i$ and $\delta_i$ measure the changes between $D_i$ and $R_i$. 

Another important parameter of interest is the mortality rate $\mu$. This parameter controls the number of subjects entering and leaving the entire population. The birth rate measures how many individuals enter into the population. The death rate measures how many members leave the population from each group. We can think of the population as isolated from other surrounding populations. New members are either born into the population or entered the population with no previous drug use. This means that new members of the population are first classified in the $S$ group. The model does not take into consideration current drug users who might enter the population, by assuming  their number  to be negligible.

An important assumption to make in this model is to assume the birth rate and the death rate  of the non-addicted  population are the same. This means that the number of persons who are born into the population and the number of those leaving the population is fixed. This follows from the assumption that the total population $N$ remains constant. More importantly, we show that the death rate is the same for each group in the population. So, since the total population is assumed to be constant,  being addicted to a drug does not increase the chances of death, in contrast to someone who has never used drugs in their lifetime. This does \emph{not} mean that the number of deaths is the same for all participating groups. The number of deaths is proportional to the size of the group. The birth rate is proportional to the size of the total population $N$.

\begin{Proposition} Assume that the total population $N$ is constant with the parameters defined above. Then the mortality rate  is constant throughout the entire model, that is, the birth rate and the death rate is the same for each group in the population.\end{Proposition}


\begin{proof} Since the total population $N$  is assumed to be constant, the same number of people are coming into and leaving the population. This implies
$$\mu N= \mu S+ \mu_{11}  D_1 +\mu_{12}  D_2 + \mu_{21}  R_1 +\mu_{22}  R_2,$$
where $\mu_{ij} \geq \mu>0$, $i$,$j=1$, $2$, as it is natural to assume that the death rate of the addicted or the recovered is at least as high as the non-addicted. Then, since $ N= S+ D_1 + D_2 + R_1 +R_2$,  we have
$$0= (\mu_{11} -\mu) D_1 +(\mu_{12} -\mu) D_2 + (\mu_{21} -\mu)  R_1 +(\mu_{22} -\mu)  R_2.$$
Since all of the involved quantities are nonnegative, this equality holds if and only if $\mu_{ij}=\mu$.
\end{proof}

\subsection{The model}

Now that the population groups and the parameters have been specified, we present and justify the population dynamics model that we propose for modeling  the addiction  in a presence of two drugs. Each of the equations below describes the rate of change in the number of individuals in each group described above per year. Members of the population first enter $S$ when they are born and then either stay in $S$ or evolve into members of one of the other groups.

\begin{table}[b]
\centering
\label{my-label}
\begin{tabular}{|ll|}
\hline
\multicolumn{2}{|c|}{\textbf{Parameters}}                               \\ \hline
$\beta_{i}$  & The influence rate of drug $i$                                   \\
$\gamma_{i}$ & The recovery rate of drug $i$                                    \\
$\delta_{i}$ & The relapse rate of drug $i$                                     \\
$\alpha_{i}$ & The rate of drug $j$ users who switch their preference to drug $i$ \\
$\mu$     & The mortality rate                                             \\ \hline
\end{tabular}
\caption{A summery of all the parameters in the full model \eqref{2} where $i,j = 1,2$. All rates are given in units of users per year.}
\end{table}

 The model can be written as 
\eq{ 
\frac{d\, S}{ d \, t}&= \mu N- \mu S-  \beta_1 \frac{S D_1}{N}-\beta_2 \frac{S D_2}{N},\\
\frac{d\, D_1}{d \, t}& = \beta_1 \frac{S D_1}{N} + \delta_1R_1 + \alpha_1 \frac{D_1D_2}{N} - \alpha_2 \frac{D_1D_2}{N}-\gamma D_1 - \mu D_1,\\
\frac{d\, D_2}{d \, t} &= \beta_2 \frac{S D_2}{N} + \delta_2R_2 + \alpha_2 \frac{D_1D_2}{N} - \alpha_1 \frac{D_1D_2}{N}-\gamma D_2 - \mu D_2,\\
\frac{d\, R_1}{d \, t} &= \gamma_1D_1-\delta_1 R_1 - \mu R_1, \\
\frac{d\, R_2}{d \, t} &= \gamma_2D_2-\delta_2 R_2 - \mu R_2.  
}{2}
We visualize the full model as a diagram in  Figure~\ref{fig:fig1}. Individuals who are born into the population are placed in $S$. Each of the five groups lose members due to the constant mortality rates. The size of $D_1$ and $D_2$ grow by taking members away from $S$. When a person first becomes addicted to a drug, they are then classified into one of these two groups. Once a person leaves the susceptible group, they can no longer be classified back into $S$. That person will remain addicted until either one of the following occurs: (1) the user dies, (2) the user recovers from their drug addiction or (3) the user decides to switch their drug preference.

Members of the population that have become addicted to a drug can either switch their preference or recover from using drugs altogether. If a person decides to switch preferences, they are now classified into the other drug group. A person can switch preferences multiple times. There is no limit to how often one person can switch. The model does not reflect those who quit using the drug they do not have a preference for. If a person prefers using drug 1, they may or may not have used drug 2 at any point. Once a person  recovers from both drugs, they then are classified as recovered. Which group they are classified in is based solely on their previous addiction. If a person relapses, they  are assumed to relapse back to their drug of choice.

\begin{figure}[b]
\scalebox{0.45}{
\includegraphics{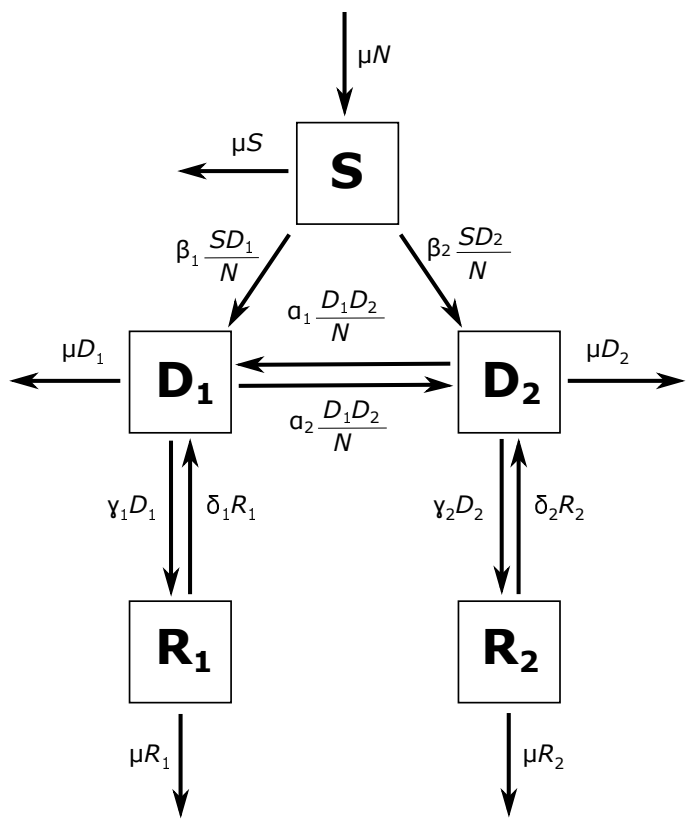}}
\caption{\label{fig:fig1}\small This is a illustration of model \eqref{2}, depicting the different population groups and their interactions. A person enters the population through $S$. The influences of $D_1$ and $D_2$  on $S$ causes those in the susceptible group to become addicted to either drug 1 or drug 2. Once a person leaves $S$, they cannot go back. When a person fully recovers from using either drug, they then recover to $R_1$  or $R_2$, depending on the last drug the person was addicted to. The flows between $D_1$  and $D_2$  signifies the change of preferences between drugs among drug users. Drug users may switch preferences as many times as they like. The parameter $\mu$ represents the natural birth rate and death rate  for each participating group.}
\end{figure}

The model in its current form is highly complex due to the number of unknown  quantities and parameters. 
The assumption that the population size is all the same reduces the number of unknowns in the system. Indeed, we substitute
$S=N-(D_1+D_2+R_1+R_2)$ into \eqref{2} to get
\eq{ 
\frac{d\, D_1}{d \, t}& = \beta_1(N-D_1-D_2-R_1-R_2) \frac{ D_1}{N} + \delta_1R_1 + \alpha_1 \frac{D_1D_2}{N} - \alpha_2 \frac{D_1D_2}{N}-\gamma D_1 - \mu D_1,\\
\frac{d\, D_2}{d \, t}& = \beta_2(N-D_1-D_2-R_1-R_2) \frac{ D_2}{N} + \delta_2R_2 + \alpha_2 \frac{D_1D_2}{N} - \alpha_1 \frac{D_1D_2}{N}-\gamma D_2 - \mu D_2,\\
\frac{d\, R_1}{d \, t} &= \gamma_1D_1-(\delta_1+\mu)R_1, \\
\frac{d\, R_2}{d \, t} &= \gamma_2D_2-(\delta_2+\mu)R_2,
}{3}
and, upon rearranging, obtain
\eq{ 
\frac{d\, D_1}{d \, t}& =(\beta_1-\gamma_1-\mu )D_1 +\delta_1R_1-\frac{D_1}{N} \left[
\beta_1 D_1 +(\alpha_2 -\alpha_1 + \beta_1)D_2+\beta_1 R_1+\beta_1 R_2 \right] ,\\
\frac{d\, D_2}{d \, t} &=(\beta_2-\gamma_2-\mu )D_2 +\delta_2R_2-\frac{D_2}{N} \left[
(\alpha_1 -\alpha_2 + \beta_2) D_1 +\beta_2 D_2+\beta_2 R_1+\beta_2 R_2 \right],\\
\frac{d\, R_1}{d \, t} &= \gamma_1D_1-(\delta_1+\mu)R_1, \\
\frac{d\, R_2}{d \, t} &= \gamma_2D_2-(\delta_2+\mu)R_2.
}{4}
This system includes four population groups so the model must be viewed as a four-dimensional dynamical system. However, as a first step in our analysis, we propose to simplify this model and reduce it to a two-dimensional dynamical system. Since the focus of the paper is to examine the interaction between the populations addicted to either drug, the model should be reduced in such a way that places emphasis on $D_1$ and $D_2$. These assumptions and the reduction of the full model are developed in the next section.

\begin{Remark}
\rm{The approach presented here uses the original definitions of $D_{i}$ and $R_{i}$. One could also non-dimensionalize the system in the following way. The total size of the population $N$ is a fixed constant with \eqref{1} as a constraint. $N$ can be factored out from the system by redefining the definitions of $D_{i}$ and $R_{i}$. By rescaling, we can write $\frac{D_{i}}{N}$ and $\frac{R_{i}}{N}$ in place of $D_{i}$ and $R_{i}$ in \eqref{2}, respectively. The importance is the relative size of $D_{i}$ and $R_{i}$ to $N$, not the actual size of $N$. Thanks to the anonymous reviewer for this suggestion.}
\end{Remark}

\section{Analysis of the Reduced Model}
 
In this section we formulate  reasonable conditions  that can be implemented to reduce the full model. The focus of this paper is to study the interaction of the populations $D_1$  and $D_2$  addicted to one of the two available  drugs, respectively. There are additional assumptions that are needed to reduce the model further but still capture the interaction between  these two groups. So, we assume that the number of members in $R_1$ and $R_2$ remain constant  throughout time, that is, $\frac{dR_1}{dt} =\frac{dR_2}{dt}=0$. Then from \eqref{4}, $R_1$  and $R_2$ can be written in terms of $D_1$ and $D_2$:
\begin{equation}R_1= \frac{\gamma_1}{\delta_1+\mu}D_1, \quad  R_2= \frac{\gamma_2}{\delta_2+\mu}D_2.\label{5} \end{equation}
This reduces the model to the two-dimensional dynamical system,
\small
\eq{ 
\frac{d\, D_1}{d \, t}& =\left (\beta_1-\mu  -\frac{\delta_1 \gamma_1 }{\delta_1+\mu}\right)D_1-\frac{D_1}{N} \left[
\beta_1 \left( \frac{\gamma_1 }{\delta_1+\mu} + 1\right)D_1+\left(\beta_1\left( \frac{\gamma_2}{\delta_2+\mu} +1\right)-\alpha_1  +\alpha_2\right )D_2\right] ,\\
\frac{d\, D_2}{d \, t} &=  \left(\beta_2-\mu  -\frac{\delta_2 \gamma_2 }{\delta_2+\mu}\right)D_2 -\frac{D_2}{N} \left[\left(
\beta_2\left( \frac{\gamma_1 }{\delta_1+\mu} + 1\right) +\alpha_1 -\alpha_2\right)D_1 +\beta_2\left( \frac{\gamma_2}{\delta_2+\mu} +1\right) D_2\right].\\
}{6}
\normalsize
We will analyze \eqref{6}, which will be referred to as the reduced model. This system incorporates the addition of a second drug while retaining the spirit of the original model \cite{Castillo}. 

One can recognize this system as the competitive Lotka-Volterra equations, which are well-known in population dynamics  as a model  that describes  time-dependent dynamics of species competing for a common resource. The system in general is very well studied, but nevertheless, since the results depend on the parameters, we here investigate  the regimes  that are of particular importance to our application.

\subsection{Fixed Points}

The next part of the analysis is to locate the fixed points of \eqref{6}. The fixed points are important as they determine the asymptotic behavior of  the solution curves  in the phase portrait, so the evolution of both addicted populations can be determined by analyzing the fixed points of the reduced system. This section shows the derivation and implication in context of each fixed point of the reduced model.

Fixed points occur whenever both equations in the system are identically zero. From inspection of \eqref{6}, it is clear that the origin is a fixed point. This outcome reflects the scenario where usage of both drugs die out. Over time, the number of users for both drugs will dwindle as the solution curves head towards the origin. Solution curves pointing away from the origin indicates an increase in the number of users for both drugs. A more detailed analysis of the origin will be discussed later in this section. 

Two more fixed points can be found by setting either $D_1=0$ and $D_2=0$. The first case yields a fixed point of
\begin{equation}\label{7}
\left(
\frac{\beta_1+ \frac{\gamma_1 \delta_1}{\delta_1+\mu} - \gamma_1  -\mu}{ \frac{\beta_1}{N}\left(1+ \frac{\gamma_1 }{\delta_1+\mu}\right)}, 0
\right),
\end{equation}
provided that $\beta_1\neq 0$. This fixed point implies that no one is addicted to drug 2, while there is a certain nonzero  number of people who prefer drug 1. 

The other fixed point  is
\begin{equation}\label{8}
\left(0, 
\frac{\beta_2+ \frac{\gamma_2 \delta_2}{\delta_2+\mu} - \gamma_2  -\mu}{ \frac{\beta_2}{N}\left(1+ \frac{\gamma_2 }{\delta_2+\mu}\right)}
\right),
\end{equation}
provided that $\beta_2\neq 0$. Similarly, this fixed point indicates that the drug 2 addiction  takes over. 

In most cases, these three fixed points will always appear in the analysis. There are some cases where a fourth fixed point might also appear. To obtain this fixed point, factor out  $D_1$ and  $D_2$, respectively, and then solve the system 
 \eq{
 N\left (\beta_1-\mu  -\frac{\delta_1 \gamma_1 }{\delta_1+\mu}\right)&=
\beta_1 \left( \frac{\gamma_1 }{\delta_1+\mu} +1\right)D_1+\left(\beta_1\left( \frac{\gamma_2}{\delta_2+\mu} +1\right)-\alpha_1  +\alpha_2\right )D_2,\\
 N \left(\beta_2-\mu  -\frac{\delta_2 \gamma_2 }{\delta_2+\mu}\right) &=\left(
\beta_2\left( \frac{\gamma_1 }{\delta_1+\mu} +1\right) +\alpha_1 - \alpha_2\right)D_1 +\beta_2\left( \frac{\gamma_2}{\delta_2+\mu} +1\right) D_2.
}{9}
The system will have a  unique fourth fixed point, provided that \eqref{9} has a unique solution, which occurs if  the following holds,
\begin{equation}
\det\begin{pmatrix}
\beta_1 \left( \frac{\gamma_1 }{\delta_1+\mu} +1\right) & \beta_1\left( \frac{\gamma_2}{\delta_2+\mu} +1\right)-\alpha_1  +\alpha_2 \\

\beta_2\left( \frac{\gamma_1 }{\delta_1+\mu} +1\right) +\alpha_1 - \alpha_2& \beta_2\left( \frac{\gamma_2}{\delta_2+\mu} +1\right) 
\end{pmatrix}\neq 0. \label{det}
\end{equation}
If the determinant in  \eqref{det} is zero, then the system \eqref{6} will either have no solutions or have infinitely many of them. Therefore there will be either four total equilibrium points in \eqref{6} or a continuum of those. 

Below, we concentrate on the analysis of the stability of the  fixed points where at least one of the components is zero.


\subsection{Jacobian Derivation}

Once the fixed points are known, the stability of the fixed points will be needed to determine where the nearby solution curves are headed. The stability of the fixed points is derived from the Jacobian of the system. We can linearize \eqref{6} by taking partial derivatives. See \cite{Lynch} and \cite{Perko} for an introduction on this method. We first take partial derivatives with respect to  $D_1$ and  $D_2$. The result is four quantities that make up the elements of the Jacobian.

\small
\eq{
&J_{11} = \left( -\frac{2\beta_1}{N}\left( \frac{\gamma_1}{\delta_1+\mu} +1\right)\right)D_1+\left(\frac{\alpha_1 -\alpha_2}{N}-\frac{\beta_1}{N}\left(1+ \frac{\gamma_2}{\delta_2+\mu}\right)\right)D_2 +\beta_1-\frac{\delta_1 \gamma_1}{\delta_1+\mu}-\gamma_1-\mu,  \\
&J_{12} = \left( \frac{\alpha_1 -\alpha_2}{N}-\frac{\beta_1}{N}\left(1+ \frac{\gamma_2}{\delta_2+\mu}\right)\right)D_1,\\  
&J_{21} =\left( \frac{\alpha_2 -\alpha_1}{N} -\frac{\beta_2}{N}\left(1+ \frac{\gamma_1}{\delta_1+\mu}\right) \right)D_2, \\
&J_{22} = \left(\frac{\alpha_2 -\alpha_1}{N}-\frac{\beta_2}{N}\left(1+ \frac{\gamma_1}{\delta_1+\mu}\right)\right)D_1 + \left( -\frac{2\beta_2}{N}\left( \frac{\gamma_2}{\delta_2+\mu} +1\right)\right) D_2 + \beta_2-\frac{\delta_2 \gamma_2}{\delta_2+\mu}-\gamma_2-\mu.
}{11}
\normalsize

The elements of the Jacobian, unlike the fixed points, include $\alpha_1$  and $\alpha_2$. The interaction parameters are only involved in the stability of the fixed points alone. They do not contribute where the fixed points will occur. Also, the relative size of  $\alpha_1$  and $\alpha_2$  does not influence the model directly. What is important is  the difference of these two values. If the switchover is not very large, then the two drugs do not have much of an influence on each other. On the other hand, if the difference is substantially large, then one of the drugs has a strong influence on the use of the other.

The eigenvalues of the Jacobian are used to classify the stability of the fixed points \cite{Lynch}. For each fixed point, a pair of eigenvalues can be obtained. If there is an eigenvalue with a positive real part then  the fixed point is unstable. This means that all of the solution curves near the fixed point travel away from the point for all time. If both eigenvalues have  negative real parts then the fixed point is asymptotically stable. The solution curves near the stable fixed point tend to go towards the point as  time progresses. An example of an unstable fixed point occurs   if the pair of eigenvalues do not have the same sign. The fixed point, in this case, is a saddle.

The stability of the fixed points are used to determine the fate of both drug populations. For example, if the fixed point given by \eqref{7} is stable, then the population will see the use of the second drug die out in due time if the population distribution is  initially close to this fixed point. If this fixed point is unstable then the population will not see this outcome occur.  For the details about the method of classifying these fixed points based on their eigenvalues, one may consult \cite{Strogatz}.

Because so many parameters are involved in our model, in some situations we calculate the eigenvalues numerically to determine both the fixed points and their stability for particular parameter regimes. Calculations in this paper are done using the {\it pplane8} application \cite{Polking}, available for MATLAB or other software. Equation \eqref{6}, as it is written, can be implemented into the application and the program can then produce solution curves for any set of parameter values.

\subsection{The stability of the origin.}

The origin is a special fixed point of the reduced model. At the origin, the number of users for both drugs is zero. If this fixed point is stable, then it  implies that there will be no drug users in the long term if there is a small number of drug users initially. Special analysis can be done on any given parameters in order to easily classify what happens at the origin.

At the origin, the Jacobian given in \eqref{11} simplifies to a diagonal matrix.
\eq{
J_{(0,0)}=\left[\begin{matrix}\beta_1+\frac{\gamma_1\beta_1}{\delta_1+\mu}- \gamma_1  -\mu  &0\\
0& \beta_2+\frac{\gamma_2\beta_2}{\delta_2+\mu}- \gamma_2  -\mu \end{matrix}
\right].
}{12}
Since the Jacobian in this case is a diagonal matrix, the eigenvalues are simply the elements along the diagonal. A similar procedure of looking at the signs of the eigenvalues, just like for the other fixed points, could be done. An alternate method of determining the stability of the origin will be derived here. This involves looking at the size of $\mu$ relative to the other given parameters.

Let $i=1,2$ denote the two drugs in the population and examine the quantity
\begin{equation}\theta_i=\beta_i+\frac{\gamma_i\beta_i}{\delta_i+\mu}- \gamma_i. \label{13}\end{equation}
If $ \theta_i<\mu $ then this implies that $\lambda_i<0$, and if $ \theta_i>\mu $ then $\lambda_i>0$. Notice once again that the stability of the origin does not rely on either $\alpha_1$  or $\alpha_2$, so the stability of the origin is not influenced by the choice of $\alpha_1$ and $\alpha_2$. This makes sense intuitively, since the fate of both drug populations should not depend on the amount of interaction between the two groups. With this setup, the origin can be generalized into three separate cases, depending on the relative size of $\mu$.

\begin{enumerate}
\item	If both $\theta_1<\mu$  and  $\theta_2 <\mu$, the origin is a stable node.
\item	If both $\theta_1 >\mu$ and $\theta_2 >\mu$, the origin is an unstable node.
\item	If either  $\theta_1 <\mu$  and $\theta_2  >\mu$   or $\theta_1  >\mu$  and $\theta_2 <\mu$, the origin is a saddle.
\end{enumerate}

In the first case, solution curves near the origin will be pulled towards the origin. This occurs because the size of $\mu$ is large relative to the size of the other parameters. The mortality rate is so large that members of the population are dying before they can even get hooked on drugs. There are not enough drug addicts to maintain both drug populations so their numbers will drop over time until the use of both drugs become extinct.

The second case is the one that is mostly likely to occur in realistic applications. The relative size of $\mu$ should be small in comparison to the other parameters. People are not dying out as rapidly, unlike in the first case. This increases the chances of an individual becoming addicted to drugs. The drug populations can be sustained by attracting new users and so the number of drug addicts should increase with time.

In the third case, the origin is a saddle point. In this situation, the eigenvectors corresponding to the eigenvalues are aligned with the coordinate axes in the phase space. Near the origin,  one of the drug populations dies out very rapidly. At the same time, addicts of the second drug increases in number.

\section{Simulation}

\subsection{Implementation}

A simulation is provided here to illustrate how the model works and what conclusions can be made from it. The parameters given in Table 1 are only used here for demonstration. In reality, these parameters would need to be derived or estimated. Consult \cite{Brauer} and \cite{Castillo} for an introduction on how these parameters can be estimated. The influence rates are found by identifying new drug addicts within the past year who say they have been persuaded by peer pressure, whether their parents have used the drugs before and whether their friends have also used the drug. These new drug users have never been addicted before in their lifetime. We use a similar idea to calculate  $\alpha_1$ and  $\alpha_2$ by examining those who currently have a drug addiction. The recovery rate is calculated by the number of people within the last year who have quit using one or both drugs for at least one year, depending on the user's last drug addiction. The relapse rate is determined by identifying those who have tried to quit from their addictions in the past year but were unsuccessful. The mortality rate is set at 0.1 and the total population size is 10,000 for this demonstration. The population size is large enough to keep the size of the population close to being constant as possible. All of these parameters are assumed to remain constant throughout the entire process. Given this information, the reduced model can be used to determine how the dynamics of the two drugs in the population will change. The dynamics will assist the experimenter in determining which drug is more likely to dominate the population and how quickly the addiction will spread. 

\renewcommand{\arraystretch}{0.7} 
\begin{table}[b]
\begin{tabular}{|ll|}\hline
\multicolumn{2} {|c|}{$\qquad$ {\bf List of parameters} $\qquad$} \\
\hline &\\ 
 $\qquad \alpha_1=0.2$ & $\qquad \alpha_2=0.3 \qquad$\\
$\qquad \beta_1=0.3$ & $\qquad \beta_2=0.5 \qquad$\\ $\qquad \delta_1=0.2$ & $\qquad \delta_2=0.3\qquad$ \\$\qquad \gamma_1=0.03$& $\qquad \gamma_2=0.04 \qquad$\\ $\qquad \mu=0.1$ & $\qquad N=10,000 \qquad$\\ &\\ \hline \multicolumn{2} {c}{$\qquad$}  
\end{tabular}
\caption{\small This is a list of parameters used for the simulation. These parameter values are only for demonstrating how the model works. In real-life applications, these parameters would need to be estimated. In this setup, the second drug is more potent than the first one. The population remains fixed at 10,000 throughout the entire process.}
\end{table}

Software is implemented to determine the fixed points, the Jacobian at each fixed point, and their stability. We use the {\it pplane8} application in MATLAB for these calculations. For an introduction of using MATLAB to analyze dynamical systems, see \cite{Lynch}. Each of the parameters in Table 2 are entered into the application and the program produces the phase portrait shown in Figure 2. The fixed points, along with their stability, can be analyzed from the phase portrait.

\begin{figure}
\scalebox{0.5}{
\includegraphics{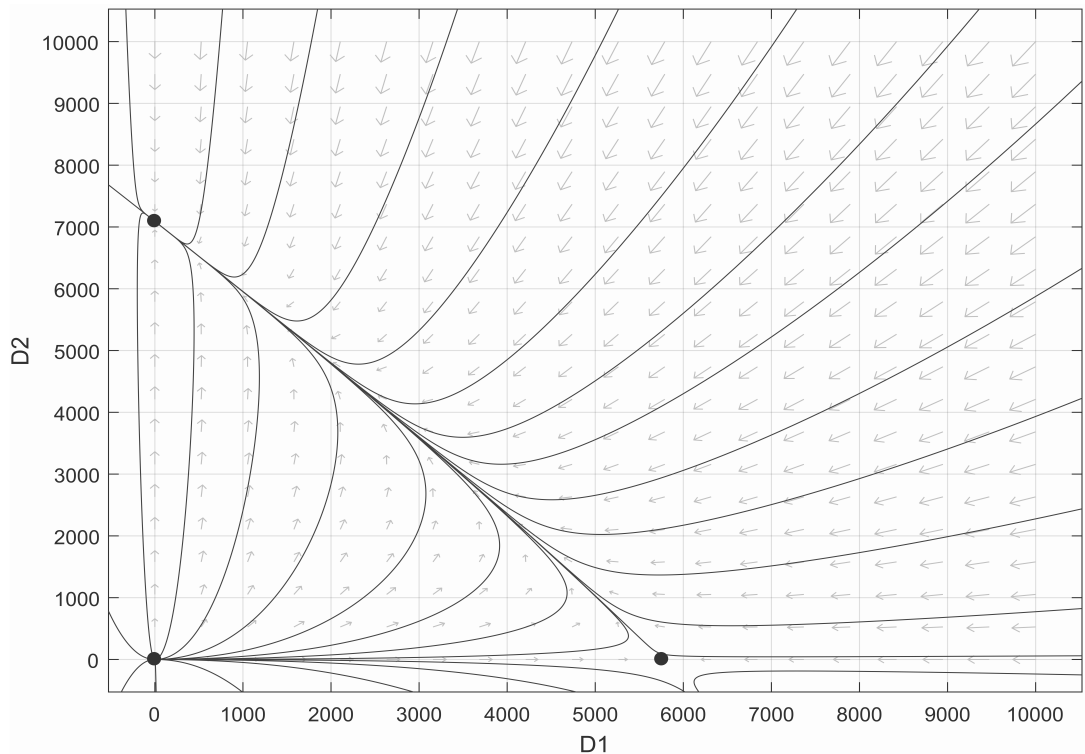}}
\caption{\label{fig:fig2}{\small \bf Phase Plane Portrait of the Simulation Study}:
\small The phase plane portrait uses the parameter values specified in Table 1. There are three fixed points, as marked above. One fixed point includes the origin (which is unstable in this case) while the other two are affixed along the  $D_1$ and $D_2$  axis. The fixed point along the $D_1$  axis is stable while the one on the $D_2$ axis is a saddle. Only a certain region of the phase plane needs to be analyzed since the total population remains fixed at 10,000.
 }
\end{figure}

The phase plane portrait in Figure \ref{fig:fig2} shows three fixed points. One fixed point is located at the origin while the other two are along the $D_1$ and $D_2$  axis. The solution curves tend to flow towards the fixed point along the $D_2$  axis, indicating that this fixed point is stable. The origin is unstable with all of the solution curves pointing away from the fixed point. The third fixed point along the $D_1$  axis is a saddle point. While some solutions flow towards the saddle, ultimately the solutions tend to move away and towards the stable node instead. 
Changing the parameters ever so slightly will adjust how the solutions are drawn. A change of say 0.1 to the influence rates will dramatically affect how each of the solution curves are drawn but may or may not affect the position of the fixed points or their stability.
Keep in mind that the population remains fixed at 10,000 throughout so the solution curves in the top-right corner of the phase plane are irrelevant. The fixed points are also not close to 10,000 but rather settle far below that threshold. While the population remains fixed for all time, the model suggests that not every person in this population will eventually become a drug user. There are those in the population who have recovered from their drug habits. Still, others will never become drug addicts in the first place. This explains why none of the fixed points ever come close to 10,000. The phase portrait also confirms that there are only three fixed points in this simulation.

There are three nullclines in Figure 2. Two lie on the $D_1$ and $D_2$ axes, respectively. These nullclines occur when either one of $\frac{d\, D_1}{d \, t}$ or $\frac{d\, D_2}{d \, t}$ is 0 \cite{Perko}, \cite{Scheinerman}. The third nullcline connects the two fixed points along the $D_1$ and $D_2$ axes. In the phase portrait, solution curves tend towards the nullcline and then swiftly approach the stable node along the $D_2$ axis. This forms a triangular region in the phase plane. Solution curves inside this region tend to see an increase in the number of drug users, regardless of drug choice. Solution curves outside this region depict situations where the number of drug users actually decreases with time before settling on the fixed point.

\subsection{Verification of the Results}

The three fixed points identified in Figure 2 can be verified numerically using Equations \eqref{7} and \eqref{8} to determine their location in the phase plane. All of the calculations done here were performed with the {\it pplane8} application in MATLAB \cite{Polking}. With the parameters given in Table 1, Equations \eqref{7} and \eqref{8} can be used to calculate the fixed point along the  $D_1$ and $D_2$ axis respectively. Therefore, it can be shown, using software or by direct calculation, the fixed point along the $D_1$ axis is (5757.576, 0) and the fixed point along the  $D_2$ axis is (0, 7090.909) in this simulation.
In order to verify the stability of each of the fixed points, the Jacobian needs to be calculated for each point. For the origin, the Jacobian calculation can be shown to be
\begin{equation*} J_{(0, 0)}=\left[\begin{matrix}0.19  &0\\0& 0.39 \end{matrix}
\right].
\end{equation*}
Note that the Jacobian at the origin is indeed a diagonal matrix. The eigenvalues are easily found by identifying the elements along the diagonal. Both eigenvalues are positive in this case which confirms that the origin is unstable. Advanced methods, as discussed in \cite{Strogatz} can be used to show that the origin is actually an unstable node. Using \eqref{13}, calculations show
 $$\theta_1=0.29, \qquad \theta_2 =0.49.$$
Since  $\mu < \theta_1$ and  $\mu< \theta_2$, the origin is an unstable fixed point. The model says that the population of drug users will not die out due to a large mortality rate, again confirming the origin being unstable.
Similarly, the Jacobian of the other two fixed points can be calculated either via software or by using direct calculations.
\begin{equation*} J_{(5757.576, 0)}=\left[\begin{matrix}-0.19  &-0.19\\0& 0.0733 \end{matrix}
\right],
\end{equation*}
\begin{equation*} J_{(0, 7090.909)}=\left[\begin{matrix}-0.044&0\\-0.39&-0.39 \end{matrix}
\right].
\end{equation*}
The Jacobian matrices here are triangular, so the eigenvalues are given by the diagonal elements. The eigenvalues for the first Jacobian matrix are  $\lambda_1=-0.19$ and  $\lambda_2=0.0733$. These eigenvalues do not share the same sign, validating that the fixed point along the  axis is indeed a saddle. The eigenvalues for the second Jacobian matrix are  $\lambda_1=-0.39$ and  $\lambda_2=-0.044$. Both of these eigenvalues are negative, confirming the fixed point along the axis is stable. It can also be shown, using advanced methods, that this fixed point is indeed a stable node.

\subsection{Interpretation}

For this simulation, the model identified three fixed points with parameter values given in Table 1. The origin marks where the population is free of the two drug addictions; no person is addicted to either drug. The origin was shown to be unstable, representing a growth in drug usage in this case. This kind of behavior is typical, as one would expect to see some drug addicts in any given society. The stable fixed point along the $D_2$ axis shows that this population will eventually become addicted to drug 2. In this case, the drug 1 population will slowly die out as more people become addicted to drug 2. How quickly this outcome occurs depends on the parameters and the initial conditions. If $\beta_2$ and $\alpha_2$ are chosen to be large, then drug 2 will quickly become the dominate drug. The saddle point along the $D_1$  is due to the initial conditions. If there are many drug 1 addicts, that group will slowly see an increase in users initially. However, their members will ultimately become drug 2 addicts due to their influence. The location of the fixed points give an estimate of how many people will be addicted to either drug in the long run. Therefore, based on the information provided for this population at this current state, drug 2 will become the dominate drug of choice over drug 1 in the foreseeable future.

\section{Conclusions}

The model proposed in this paper is an extension to the model provided by Castillo-Garsow, Jordan-Salivia, and Rodriguez-Herrera \cite{Castillo} by incorporating two drugs in a population. In this setup, the option is there for one drug to have a larger influence than the other one. This model demonstrates what happens when these kinds of situations occur. The larger the influence on the susceptible group, the more people become addicted to that type of drug. In the model however, the larger the influence a particular drug has on the other one, the more likely people are to switch their preference between drug habits. In the end, the switchover rate has a much stronger effect than the influence rate only in determining which drug will become dominant. The switchover rate plays no role in how many drug users there will be or whether both drug populations will manage to coexist.

The analysis shows what happens over time in a society given two different kinds of drugs. The fixed points indicate what will happen in the population in the long run. Given a set of initial conditions, usually observed in practice, one can predict which drug addiction will dominate. One outcome leads to a drug-free population where there are no active drug addicts. Another population may see individuals addicted to both drugs that are readily available. Still, one drug addiction overtaking  another drug addiction in a population is also a possibility. With the outcome of a population known, preparations can be made to target users who are addicted to the more popular drug.

The model can be extended to cover certain special cases, some of which will be mentioned here. One application is to let one of $\alpha_i$ be zero. In this setup, users cannot switch their preference back to drug $i$. Once a person becomes addicted to one drug, they cannot switch their preference back to the previous one. Users of the ``harder'' drug are left with only two options: recover from the addiction or leave the population altogether.  Another special case involves one drug having a very strong influence on the other. In this situation, $ \alpha_1 \ll\alpha_2$ and $ \beta_1 \ll\beta_2$. Typically in these situations, one minor drug addiction leads to a more serious one. While it is possible to revert preference back to the minor drug, cases like these are rare in these situations. 

The model can also be easily applied to situations where two diseases are present in a population. This idea reflects upon the Epidemic SIR Model of Kermack and McKendrick \cite{Brauer}. Let $D_i$ denote those in a population who are infected with disease $i$. In situations where a disease may be fatal, let $\gamma_i=\delta_i=R_i=0$ and define the $\alpha$s appropriately. In this setting, persons infected with the fatal disease cannot be cured in the population. Those infected must live with the disease until they leave the population.

Depending on the initial values of the parameters, it may be possible to force the system to behave in such a way that other equilibria appear. This paper examines three of the fixed points but the possibility is there for more fixed points to exist. The existence of more than three fixed points indicates the presence of a coexisting fixed point. The stability of the three fixed points derived here will, more than likely, change to accommodate the new fixed points. Investigating coexisting fixed points could be used as a topic of future research and insight.

\section{Acknowledgments}
The authors are grateful to the anonymous reviewer for detailed comments and many suggestions that significantly improved this paper.

This project  was  developed under the guidance of A. Ghazaryan and was supported by the National Science Foundation through grants DMS-1311313 (to A. Ghazaryan).

Stephen Colegate earned his M.S. in Statistics at Miami University in Oxford, Ohio. His research interests include dynamical systems and population models, time series, Bayesian inference. As an instructor at Xavier University in Cincinnati, Ohio, Stephen loves to share his passion of mathematics and statistics with his students. His hobbies include reading, bicycling, quad-riding, and traveling to new places. Stephen also enjoys spending time with his family (Louie, Sheila, and Zachary) and is a member of Hamilton Christian Center in Hamilton, Ohio.

Changrui (Charlie) Liu is a graduate student at Miami University in Oxford, Ohio, majoring in mathematics. His research areas including dynamical systems, differential equations, complex analysis and probability theory. He is fond of swimming, playing pianos and also traveling.
\end{document}